\documentclass[12pt]{article}
\usepackage{amsmath, amsfonts, amsthm, amssymb, epsfig, textcomp}

\newtheorem{theorem}{Theorem}

\begin{document}
\title{{\bf The Diophantine Equation}\\
\ \ \ \ \\ $\mathbf{\arctan \left(\dfrac{1}{x}\right) +
    \arctan \left(\dfrac{\ell}{y}\right) = \arctan \left(\dfrac{1}{k}\right)}$}
\author{Konstantine Zelator\\
Mathematics, Statistics, and Computer Science\\
212 Ben Franklin Hall\\
Bloomsburg University of Pennsylvania\\
400 East 2nd Street\\
Bloomsburg, PA  17815\\
USA\\
and\\
P.O. Box 4280\\
Pittsburgh, PA  15203\\
kzelator@bloomu.edu\\
e-mails: konstantine\underline{\ }zelator@yahoo.com}

\maketitle

\section{Introduction}  The subject matter of this work is the two-variable
diophantine equation $\arctan\left(\dfrac{1}{x}\right) +
\arctan\left(\dfrac{\ell}{y} \right) = \arctan \left(\dfrac{1}{k}\right)$ for
given positive integers $k$ and $\ell$, such that gcd $(\ell,k^2+1) =1$ (i.e.,
$\ell$ and $k^2+1$ are relatively prime).  The main objective is to determine
all positive integer pairs $(x,y)$ which satisfy

\begin{equation}
\left\{\begin{array}{l} \arctan\left(\dfrac{1}{x}\right) + \arctan
\left(\dfrac{\ell}{y} \right) = \arctan\left(\dfrac{1}{k}\right)\\
x,y \in \mathbb Z^+,\ {\rm gcd}(\ell,k^2+1)=1\ {\rm and} \\
{\rm with\ gcd}\ (\ell, y) =1\ \left({\rm i.e.,}\ \ell\  {\rm and}\ y\  
{\rm are}\right.\\
\left.{\rm relatively \ prime}\right) 
\end{array} \right\} \label{E1}
\end{equation}

\noindent This is done in Theorem 1, Section 4.  As we will see, there are
exacgtly $N$ distinct solutions to (\ref{E1}) where $N$ is the number of
positive divisors of the integer $k^2+1$.  The $N$ pairs $(x,y)$, which are
solutions to (\ref{E1}), are expressed parametrically in terms of the positive
divisors of $k^2+1$.  Also, note that when $\ell =1$, equation (\ref{E1}) is
symmetric with respect to the two variables $x$ and $y$.  If $(a,b)$ is a
solution, then so is $(b,a)$.  The motivating force behind this work is a
recent article published in the journal {\it Mathematics and Computer
  Education} (see \cite{1}).  The article, authored by Hasan Unal, is entitled
``Proof without words: an arctangent equality''.  It consists of four
illustrations, a purely geometric proof of the equality,

$$\arctan\left(\dfrac{1}{3} \right) + \arctan\left(\dfrac{1}{7}\right) =
\arctan \left( \dfrac{1}{2}\right).
$$

\noindent From the point of view of (\ref{E1}), the last equality says that
the pair $(3,7)$ is a solution of (\ref{E1}), in the case $\ell=1$ and $k=2$.

According to Theorem 1, $(3,7)$ and $(7,3)$ are the only solutions to
(\ref{E1}) for $\ell =1$ and $k = 2$. 

This, then, is the other objective of this article.  To generate more
arctangent type of equalities.  This is done in Section 5, where a listing of
such equalities is offered; an immediate consequence of Theorem 1.

In Section 2, we list two trigonometric preliminaries: the well known
identity for the tangent of the sum of two angles and a couple of basic facts
regarding arctangent function.  

In Section 3, we state two well known results from number theory: Euclid's
lemma; and  the formula that gives the number of positive divisors of a
positive integer. We use these in the proof of Theorem 1.

\section{Trigonometric preliminaries}

\begin{enumerate}
\item[(a)]  If $\theta_1$ and $\theta_2$ are two angles measured in radians,
  such that neither $\theta_1 $ nor $\theta_2$, nor their sum $\theta_1 +
  \theta_2$ is of the form $k\pi + \dfrac{\pi}{2},\ k$ and integer.  Then,

$$
\tan (\theta_1+\theta_2) = \dfrac{\tan \theta_1 + \tan \theta _2}{1-\tan
  \theta_1\theta_2}
$$

\item[(b)]  Let $f$ be the arctangent function, $f(x) = \arctan x$.  Then,

\begin{enumerate} 

\item[(i)] $\arctan 1 = \dfrac{\pi}{4}$

\item[(ii)]  $\left\{ \begin{array}{lr} 0 < \theta < \dfrac{\pi}{r} & \\ 
{\rm and} & \\
& \theta = \arctan c\end{array}\right\} \Leftrightarrow \left\{
\begin{array}{l} 0 < \theta < \dfrac{\pi}{4} \\
\\
0 < c = \tan \theta < 1\end{array} \right\}$

\item[(iii)]  $\left\{ \begin{array}{ll}  0 < \theta < \dfrac{\pi}{2} \\
{\rm and} & \\
& \theta = \arctan c\end{array}\right\}  \Leftrightarrow \left\{
\begin{array}{l} 0 < \theta < \dfrac{\pi}{2} \\
\\
0 < c = \tan \theta \end{array}\right\}$

\end{enumerate}

\end{enumerate}

\section{Number theory preliminaries}

The following result is commonly known as Euclid's lemma, and is of great
significance in number theory.

\vspace{.15in}

\noindent {\bf Result 1 (Euclid's lemma):}  {\it Let $a,b,c$ be positive
 integers such that $a$ is a divisor of the product $bc$; and with $a$ also
  being relatively prime to $b$. Then, $a$ is a divisor of $c$.}

\vspace{.15in}

The next result provides a formula that gives the exact number of positive
divisors of a positive integer.

\vspace{.15in}

\noindent {\bf Result 2 (number of divisors formula)}  {\it Let $n\geq 2$ be a
  positive integer, and let $p_1  , \ldots , p_t$ in increasing order, be the
  distinct prime bases that appear in the prime factorization of $n$, so
  that $n= p^{e_1}_1, \ldots p^{e_t}_t$, with the exponents $e_1, \ldots , e_t$
  being positive integers.  Also, let $N$ be the number of positive divisors of
  $n$.  Then,

\begin{enumerate} \item[(i)] $N = \overset{t}{\underset{i=1}{\Pi}} (e_i+1)
  \ldots (e_1+1) \ldots (e_t + 1)$.
\item[(ii)] In particular, when $e_1 = \ldots = e_t =1$ (i.e., when $n$ is
  squarefree)
\end{enumerate}}

$$N=2^t
$$

\noindent Both of these two results can be easily found in number theory books
and texts.  For example, see reference \cite{2}.

\section{Theorem 1 and its proof}

\begin{theorem}  Let $k$ and $\ell$ be fixed or given positive integers such
  that \linebreak
gcd$(\ell,k^2+1)=1$.  Consider the diophantine equation (\ref{E1}).

\begin{enumerate}
\item[(a)] There are exactly $N$ distinct positive integer pairs $(x,y)$ which
  are solutions to equation (\ref{E1}) where $N$ is the number of positive
  integer divisors of the integer $k^2+1$.  Specifically, if $(x,y)$ is a
  positive integer solution of (\ref{E1}), then

$x=k+\ell \cdot \left(\dfrac{k^2+1}{d}\right)\ {\rm and}\ y = k\ell + d$ where
  $d$ is a positive integer divisor of $k^2+1$.

\item[(b)] If $k^2+1 = p$, a prime number, then equation (\ref{E1}) has exactly
  two distinct positive integer solutions.  These are

$$(x,y) = (k + \ell (k^2+1),\ k\ell + 1),\ \ (k+\ell,\  k\ell + k^2 +1).
$$

\item[(c)]  If $k^2+1 = p_1p_2$, a product of two distinct primes $p_1$ and
  $p_2$,  equation (\ref{E1}) has exactly four distinct positive integer
  solutions.  These are,

$$\begin{array}{ll}
(x,y)=(k+\ell(k^2+1),k\ell +1), & (k+\ell,\ k\ell + k^2 +1 ),\\
\\ 
(k+\ell p_2.\ k\ell + p_1), \ {\rm and} & {\and}\ (k+\ell p_1,\ k\ell + p_2)
\end{array}
$$
\end{enumerate}
\end{theorem}

\begin{proof} First note that parts (b) and (c) are immediate  consequences of
  part (a) and Result 2.  We omit the details.  We prove part (a)

\begin{enumerate}
\item[(a)]  Let $d$ be a positive integer divisor of $k^2+1$.  We will show
  that the positive integer pair, $(x_d,y_d) = \left(k + \ell \cdot
  \left(\dfrac{k^2+1}{d}\right),\ k\ell +d \right)$ is a solution to
  (\ref{E1}).  First note that $y_d = k\ell + d$, is relatively prime to
  $\ell$.  Indeed, if $y_d$ and $\ell$ had a prime factor $q$ in common then
  $q$ would divide $y_d - k\ell = d$; and thus (since $d$ is a divisor of
  $k^2+1$) $y_d - k\ell = d$, then $q$ would divide $k^2+1$ contrary to the
  hypothesis that gcd$(\ell, k^2+1)=1$.  Thus, gcd$(\ell, y_d)=1$.  

It is clear that since $k, \ell$ and $d$ are positive integers, we have
\linebreak
$x_d > 1, y_d > 1 $ and $k \geq 1$.  So, 

\begin{equation}\left(0 < \dfrac{1}{x_d} < 1,\   0 < \dfrac{\ell}{y_d} < 1,\   0< \dfrac{1}{k} \leq 1 \right). \label{E2}
\end{equation}

Let 

\begin{equation}
\theta_1 = \arctan \left(\dfrac{1}{x_d}\right),\ \theta_2 = \arctan
 \left(\dfrac{1}{y_d}\right),\ \theta = \arctan \left(\dfrac{1}{k} \right).
 \label{E3}
\end{equation}

\noindent Then, by (\ref{E2}), (\ref{E3}) and part (b) of the trigonometric
preliminaries, we have 

\begin{equation}
\left\{\begin{array}{l} 0 < \theta_1 < \dfrac{\pi}{4},\ 0 < \theta_2 <
  \dfrac{\pi}{4},\ 0 < \theta \leq \dfrac{\pi}{4}\\
\\
{\rm and} \ 0 < \theta_1 + \theta_2 < \dfrac{\pi}{2},\ \tan \theta_1 =
\dfrac{1}{x_d},\ \tan \theta_2 = \dfrac{\ell}{y_d},\ \tan \theta = \dfrac{1}{k} \end{array}\right\} \label{E4}
\end{equation}

\noindent From (\ref{E4}) and part (a) of trigonometric preliminaries, it
follows that 

\begin{equation}
\begin{array}{rcl} \tan(\theta_1 + \theta_2) & = & \dfrac{\frac{1}{x_d} +
    \frac{\ell}{y_d}}{1- \frac{1}{x_d} \cdot \frac{\ell}{y_d}};\\
\\
\tan(\theta_1 + \theta_2) & = & \dfrac{y_d+\ell x_d}{x_dy_d - \ell};\\
\\
\tan(\theta_1 + \theta_2) & = & \dfrac{d\cdot (y_d + \ell x_d)}{dx_d y_d-
  d\ell}.\end{array}
\label{E5}
\end{equation}

\noindent By (\ref{E5}) and the expressions for $x_d$ and $y_d$ (see beginning
of the proof) we get 

\begin{equation}
\begin{array}{rcl} \tan (\theta_1 + \theta_2) & = & \dfrac{d^2 +k\ell d + k
\ell d + \ell^2 \cdot (k^2+1)}{[dk + \ell(k^2+1)] (k\ell + d) - d\ell};\\
\\
\tan(\theta_1 + \theta_2) & = & \dfrac{d^2 + 2k\ell d + \ell^2 \cdot
  (k^2+1)}{d\ell k^2 + k\ell^2 (k^2+1) + kd^2 + \ell dk^2 + d\ell - d\ell};\\
\\
\tan(\theta_1+\theta_2) & = & \dfrac{ d^2 + 2k\ell d + \ell^2 \cdot (k^2+1)}{k
  \cdot [2dk\ell + d^2 + \ell^2(k^2+1)]} = \dfrac{1}{k} = \tan \theta;\\
\\
\tan (\theta_1+\theta_2) & = & \tan \theta \end{array}\label{E6}
\end{equation}

\noindent By (\ref{E6}) and part (b) of the trigonometric preliminaries, it
follows that $\theta_1 + \theta_2 = \theta$, which combined with (\ref{E3}),
clearly establishes that the pair $(x_d,y_d)$ is a solution to (\ref{E1}).

Now, the converse.  Suppose that $(x,y)$ is a positive integer solution to
(\ref{E1}).

\noindent Then

\begin{equation} \left( 0 < \dfrac{1}{x} \leq 1,\ \ 0 < \dfrac{\ell}{y} \leq
  \ell,\ \ 0 < \dfrac{1}{k} \leq 1\right) \label{E7}
\end{equation}

Using (\ref{E7}), the trigonometric preliminaries, parts (a) and (b) and by
taking tangent of both sides of (\ref{E1}), we obtain,

$$
\dfrac{\frac{1}{x} + \frac{\ell}{y}}{1- \frac{1}{x} \frac{\ell}{y}} =
\dfrac{1}{k}
$$

\noindent or equivalently

\noindent (Note that since $0 < \dfrac{1}{k} \leq 1$.  The equal sides of (1)
can be utmost equal to $\dfrac{\pi}{4}$ )

\begin{equation} \begin{array}{rcl} xy-\ell & = & k(y+x\ell);\\
\\
y\cdot (x-k) & = & \ell\cdot (1+kx)
\end{array}
\label{E8}
\end{equation}

Equation (\ref{E8}) shows that $y$ is a divisor of the product $\ell(1+kx)$.
But, by (\ref{E1}), we know that gcd$(\ell , y) = 1$.  Thus, by Result 1
(Euclid's lemma), it follows that $y$ must divide $1+kx$; and so,

\begin{equation}
\left\{ \begin{array}{l} 1+kx = y \cdot v \\
\\
v \ {\rm a\ positive \ integer} \end{array}\right\}\label{E9}
\end{equation}

\noindent By (\ref{E9}) and (\ref{E8}) we have that,

\begin{equation} x = \ell \cdot v + k \label{E10}
\end{equation}

\noindent From (\ref{E9}) and (\ref{E10}) we further get 

$$ 1+k(\ell v + k) = yv;$$

\noindent or equivalently

\begin{equation} k^2+1 = (y-\ell k ) \cdot v \label{E11}
\end{equation}

Since $v$ is a positive integer, equation (\ref{E11}) shows that $(y-\ell k)$
is a positive integer divisor of $k^2+1$.  Let $y - \ell k = d$,  $d$ a
positive divisor of $k^2+1$.  Then $y = \ell k + d$ and by (\ref{E11}) and
(\ref{E10}) we also get

$$
x = k+ \ell \cdot \left( \dfrac{k^2+1}{d} \right),$$

\noindent which proves that the solution $(x,y)$ has the required form.  

Finally, we see by inspection that the $N$ (number of positive divisors of
$k^2+1$) positive integer solutions to (\ref{E1}) are distinct since,
obviously, all the $Ny$-coordinates are distinct.  The proof is
complete. 
\end{enumerate}
\end{proof}

\section{ A listing of nine equalities}

Let $k$ and $\ell$ be positive integers such that gcd$(\ell, k^2+1)=1$.
Applying Theorem 1 with $d =1$ and $d=k^2+1$ produces two inequalities.

\begin{enumerate} \item[1.] $\arctan \left( \dfrac{1}{k+\ell (k^2+1)}\right) +
  \arctan \left(\dfrac{\ell}{k\ell +1}\right) = \arctan \left(\dfrac{1}{k}
  \right)$

\item[2.]  $\arctan \left( \dfrac{1}{k+\ell} \right) + \arctan \left(
  \dfrac{\ell}{k\ell + k^2+1}\right) = \arctan \left(\dfrac{1}{k}\right)$
\end{enumerate}

Next, applying Theorem 1 with $k = \ell=1$, produces the equality:

\begin{enumerate}\item[3.] $\arctan \left(\dfrac{1}{3}\right) + \arctan
  \left(\dfrac{1}{2} \right) = \dfrac{\pi}{4}$
\end{enumerate}

\noindent For $\ell = 1$ and $k =2$:

\begin{enumerate}
\item[4.]  $\arctan \left(\dfrac{1}{3}\right) + \arctan \left(\dfrac{1}{7}
  \right) =\arctan \left( \dfrac{1}{2}\right)$.

\end{enumerate}

\noindent For $\ell = 1$ and $k = 3$

\begin{enumerate}
\item[5.]  $\arctan \left( \dfrac{1}{11} \right) + \arctan \left( \dfrac{1}{4}
  \right) = \arctan \left( \dfrac{1}{3} \right)$

\item[6.] $\arctan \left(\dfrac{1}{8} \right) + \arctan \dfrac{1}{5} =
  \arctan \left(\dfrac{1}{3} \right) $
\end{enumerate}

\noindent For $\ell =2$  and $k = 4$:

\begin{enumerate}
\item[7.]  $\arctan \left(\dfrac{1}{38}\right) + \arctan \left(\dfrac{2}{9}
  \right) = \arctan \left(\dfrac{1}{4}\right)$

\item[8.]  $\arctan \left(\dfrac{1}{6} \right) + \arctan
  \left(\dfrac{2}{25}\right) = \arctan \left(\dfrac{1}{4} \right)$
\end{enumerate}

\noindent For $\ell = 1$ and $k = 6$:

\begin{enumerate}
\item[9.]  $\arctan \left( \dfrac{1}{43} \right) + \arctan
  \left(\dfrac{1}{7}\right) = \arctan \left( \dfrac{1}{6}\right)$
\end{enumerate}

\end{document}